\documentclass[12pt]{article}
\usepackage{amsmath}
\usepackage{amsfonts}
\usepackage{amssymb}
\usepackage{amsmath}
\usepackage{amsthm}
\usepackage{latexsym}
\usepackage{hyperref}
\usepackage{graphicx}
\usepackage{epstopdf}
\usepackage{float}
\usepackage{amstext,amsbsy,amscd,mathrsfs,graphics,amsxtra}
\usepackage[english]{babel}
\usepackage{latexsym}
\usepackage[english]{babel}
\usepackage{color}

\newcommand{\R}{\mathbb R}

\newtheorem{theorem}{Theorem}[section]

\newtheorem{definition}[theorem]{Definition}

\newtheorem{lemma}[theorem]{Lemma}

\newtheorem{proposition}[theorem]{Proposition}
\newtheorem{remark}[theorem]{Remark}

\begin{document}
\title{{{{On flows associated to Tanaka's SDE and related works}}}}
\date{}

\maketitle
\begin{center}
\renewcommand{\thefootnote}{(\arabic{footnote})}
\scshape Hatem Hajri\footnote{Universit\'e Paris Ouest Nanterre La
D\'efense, Laboratoire Modal'X. Email: hatem.hajri.fn@gmail.com\newline This research was partially supported by the National
Research Fund, Luxembourg, and cofunded under the Marie Curie
Actions of the European Comission (FP7-COFUND).}
\renewcommand{\thefootnote}{\arabic{footnote}}\setcounter{footnote}{0}
\end{center}
\begin{abstract}
We review the construction of flows associated to Tanaka's SDE from \cite{MR2235172} and give an easy proof of the classification of these flows by means of probability measures on $[0,1]$. Our arguments also simplify some proofs in the subsequent papers \cite{MR2835247,MR50101111,MR000,MR0000000}. 
\end{abstract}


\section{Introduction}
Our main interest in this paper is Tanaka's equation  
\begin{equation}\label{oui}
X^x_t=x+\int_{0}^{t}\text{sgn}(X^x_s) dW_s
\end{equation}
where $W$ is a standard Brownian motion and $x\in\R$. Following \cite{MR2235172}, our purpose here is to determine the set 
$$\mathscr P =\{P^{\infty}=(P^n)_{n\ge 1}\}$$
 where each $P^{\infty}=(P^n)_{n\ge 1}$ is a compatible \cite{MR2060298} family of Feller semigroups acting respectively on $C_0(\R^n)$ such that for each $n$ the coordinates of the $n$ point motion (Markov process) associated to $P^n$ are solutions of (\ref{oui}) driven by the same Brownian motion $W^n$. The compatibility means that for all $k\le n$, any $k$ coordinates of an $n$ point motion (associated to $P^n$) form a $k$ point motion and for any permutation $\sigma$ of $\{1,\cdots,n\}$, if $(X^1,\cdots,X^n)$ is the $n$ point motion starting from $(x_1,\cdots,x_n)$, then $(X^{\sigma(1)},\cdots,X^{\sigma(n)})$ is the $n$ point motion starting from $(x_{\sigma(1)},\cdots,x_{\sigma(n)})$.

%
There is a one to one correspondance (see Section \ref{reminder}) between $\mathscr P$ and the set 
$$\mathscr K=\bigg\{\text{Law of}\ K : K\ \text{is a solution of the generalized Tanaka's SDE}\bigg\}$$
where the generalized Tanaka's SDE is defined as follows 
\begin{definition}\label{couche}
Let $K$ be a stochastic flow of kernels and $W$ be a real white noise. We say that $(K,W)$ is a solution of (the generalized) Tanaka's SDE $(T)$ if for all $s\leq t, x\in\R$, $f\in C^2_b(\mathbb R)$ ($f$ is $C^2$ on $\R$ and $f',f''$ are bounded), a.s.
\begin{equation}\label{dom}
K_{s,t}f(x)=f(x)+\int_s^tK_{s,u}(f'{\textrm{sgn}})(x)dW_{s,u}+\frac{1}{2}\int_s^t K_{s,u}f''(x)du
\end{equation}
We say that $K$ is a Wiener solution of $(T)$ if for all $s\le t$, $\mathcal F^K_{s,t}=\sigma(K_{u,v}, s\le u\le v\le t)$ is contained in $\mathcal F^W_{s,t}=\sigma(W_{u,v}, s\le u\le v\le t)$. Any flow of kernels $K$ solving $(T)$ that is not Wiener, will be called a weak solution.
\end{definition}
We recall that $(W_{s,t})_{s\le t}$ is a real white noise if there exists a (unique) Brownian motion on the real line $(W_t)_{t\in\R}$ such that $W_{s,t}=W_t - W_s$ for all $s\le t$. As a consequence of Definition \ref{couche}, if $(K,W)$ solves $(T)$, then $\mathcal F^W_{s,t}\subset\mathcal F^K_{s,t}$ for all $s\le t$ (see Lemma 3.1 in \cite{MR2235172}) and so we may just say $K$ solves $(T)$.

The main result of \cite{MR2235172} shows that $\mathscr K$ is in bijection with  
$$\mathscr M = \left\{m\in \mathcal P([0,1]), \int_{0}^1 x\ dm(x)=\frac{1}{2}\right\}$$
where $\mathcal P([0,1])$ is the set of all probability measures on $[0,1]$. This is summarized in the following
\begin{theorem}\cite{MR2235172}\label{a}  
\begin{itemize}

\item[(1)]  Let $m$ be a probability measure on $[0,1]$ with mean $1/2$. There exist a stochastic flow of kernels (unique in law) $K^{m}$ and a real white noise $W$ such that $(K^{m},W)$ solves $(T)$ and such that if $W_{s,t}^+=W_{s,t}-\displaystyle\inf_{u\in[s,t]}W_{s,u}$ and $\tau_{s}(x)=\inf\{r\geq s:\ W_{s,r}=-|x|\}$, then for all $s\le t$ and $x\in\R$, a.s.
$$K^{m}_{s,t}(x)=\delta_{x+\textrm{sgn}(x)W_{s,t}}1_{ \{t\leq \tau_{s}(x)\}}+(U_{s,t}\delta_{W_{s,t}^{+}}+(1-U_{s,t})\delta_{-W_{s,t}^{+}}) 1_{\{t> \tau_{s}(x)\}}$$
where for each $s<t$, $U_{s,t}$ is independent of $W$ and has for law $m$.
\item[(2)]  
$$K_{s,t}(x)=\delta_{x+\textrm{sgn}(x)W_{s,t}}1_{ \{t\leq \tau_{s}(x)\}}+\frac{1}{2}(\delta_{W_{s,t}^{+}}+\delta_{-W_{s,t}^{+}}) 1_{ \{t> \tau_{s}(x)\}}$$
is the unique (up to modification) Wiener solution of $(T)$; it corresponds to $m=\delta_{\frac{1}{2}}$. For $m=\frac{1}{2}(\delta_0 + \delta_1)$, $K^m$ is induced by a coalescing flow of mappings $\varphi^c$ i.e. $K^m=\delta_{\varphi^c}$. Moreover $\varphi^c$ is the law unique stochastic flow of mappings such that for all $s\le t$ and $x\in\R$ a.s. 
$$\varphi^c_{s,t}(x) = x + \int_{0}^{t} \text{sgn}(\varphi^c_{s,u}(x)) dW_{s,u}.$$

\item[(3)]  For any flow $K$ solution of $(T)$, there exists a unique probability measure $m$ on $[0,1]$ with mean  $1/2$ such that $K\overset{law}{=}K^{m}$.
\end{itemize}
\end{theorem}
In all this paper, a stochastic flow of kernels $K$ on $\R$ is defined as in \cite{MR2060298} with the additional assumption that $(s,t,x,\omega)\mapsto K_{s,t}(x,\omega)$ is measurable from $\{(s,t,x,\omega), s\leq t, x\in\R, \omega\in\Omega\}$ into $\mathcal P(\R)$. This is important for the integrals in (\ref{dom}) to be defined.

Let us now describe the content of this paper and our contributions to the study of $(T)$. In Section \ref{reminder}, we explain the correspondance between Feller compatible solutions to (\ref{oui}) and flows solutions to $(T)$. The content of this section was implicit in \cite{MR000} and \cite{MR0000000}. We write it here since our proofs later will strongly rely on it. In Section \ref{tou}, we briefly review the construction of the flows $K^m$ from \cite{MR2235172}, thus sketching the proof of Theorem \ref{a} $(1)$. In Section \ref{tft}, we prove Theorem \ref{a} $(2)$. The idea to establish this fact for $(T)$ and other varieties of it which appeared in \cite{MR2835247,MR50101111,MR000,MR0000000} was to show that whenever $K$ is a Wiener solution, the Wiener chaos expansion of $K_{0,t}f(x)$ is unique for $f$ bounded and smooth enough (this was not explicitly written in \cite{MR2235172}). This idea was not easy to formulate especially when the one point motion is the Walsh process (see \cite{MR2835247}). The arguments used here, based only on martingale problems, are elementary; they apply for all examples considered in \cite{MR2835247,MR50101111,MR000,MR0000000,MR501878} and also give the explicit chaos expansion of Wiener solutions in an easy way at least in the Brownian case (see Remark \ref{mama}). They may also be applied for equations driven by noises with a less obvious theory of chaos expansion for example L\'evy processes. The unicity in law of $\varphi^c$ will not be reproduced here (see Lemma 3.3 in \cite{MR2235172}). In Section \ref{tr}, we present an easy proof of Theorem \ref{a} $(3)$. The original proof of \cite{MR2235172} is based on a non trivial progressive modification of the flow and a deep result on Brownian filtrations (Lemma 4.11 \cite{MR2235172}). Here, we first reformulate this statement in terms of the compatible family of Feller semigroups associated to $K$. The proof is then achieved by means of the skew Brownian motion. The motivation behind the use of this process is the following observation: If $(X,Y)$ is the two point motion associated to $\varphi^c$ (resp. $K^W$) starting from $(0,0)$, then $\text{sgn}(X_t) Y_t$ is a reflecting (resp. standard) Brownian motion. For any solution $K$, we prove the natural conjecture that $\text{sgn}(X_t) Y_t$ is a skew Brownian motion. Its skewness parameter measures how far $\text{sgn}(X_t)$ and $\text{sgn}(Y_t)$ are from each other and in this way we are able to classify the laws of all the two point motions. Our arguments also apply for the works \cite{MR2835247,MR50101111,MR0000000,MR501878}. Finally in Section \ref{trrr}, we write down the generators of the $n$-point motions associated to the flows $K^m$ and show their dependence on $m$.
\section{Stochastic flows of kernels and weak solutions}\label{reminder}
In this section, we will explain the link between the usual Tanaka's equation (\ref{oui}) and the generalized one $(T)$. The main result is Proposition \ref{propop} below. It shows  that each $P^{\infty}\in\mathscr P$ corresponds to a unique stochastic flow of kernels solution of $(T)$ and vice versa. The second half of the proposition is stated only for completeness and will not be used in this paper. 
\begin{proposition}\label{propop}
\begin{itemize}

\item[(1)] Let $K$ be a solution of $(T)$ and let $(P^n)_n$ be its associated compatible family of Feller semigroups given by $P^n_t=E[K^{\otimes n}_{0,t}]$. Then for all $x_1,\cdots,x_n\in \R$: if $X=(X^1,\cdots,X^n)$ is the $n$-point motion associated to $P^n$ started from $(x_1,\cdots,x_n)$ and defined on $(\Omega^n,\mathcal A^n,\mathbb P^n)$, there exists an $(\mathcal F_t^X)_t$ Brownian motion $W^n$ (on the same probability space) such that for all $i$,
\begin{equation}\label{der}
X^i_t=x_i + \int_{0}^{t}\text{sgn}(X^i_s) dW^n_s.
\end{equation}
\item[(2)] Let $(P^n)_n$ be a compatible family of Feller semigroups acting respectively on $C_0(\R^n)$ such that for each $n$ and $x_1,\cdots,x_n\in \R$: if $X=(X^1,\cdots,X^n)$ is the Markov process associated to $P^n$ and started from $(x_1,\cdots,x_n)$ defined on $(\Omega^n,\mathcal A^n,\mathbb P^n)$, there exists an $(\mathcal F_t^X)_t$ Brownian motion $W^n$ (on the same probability space) such that for each $i$, $(X^i,W^n)$ satisfies (\ref{der}). Then there exist a stochastic flow of kernels $K$ and a real white noise $W$ such that $(K,W)$ solves $(T)$ and moreover $P^n_t=E[K_{0,t}^{\otimes n}]$.

\end{itemize}
\end{proposition}
\begin{proof}
(1) Let $(K,W)$ be a solution of $(T)$ defined on $(\Omega,\mathcal A,\mathbb P)$ and let $(W_t)_{t\in\R}$ be the unique Brownian motion on the real line such that $W_{s,t}=W_t - W_s$ for all $s\le t$. For $n\ge 1$, $t\ge 0$, $f\in C_0(\R^n), g\in C_0(\R)$ and $(x,w)\in\R^n\times\R$, define
$$Q^n_t(f\otimes g)(x,w)=E[K^{\otimes n}_{0,t} f(x) g(w+W_t)].$$
It is elementary to check that $Q^n$ is a Feller semigroup for all $n$. These semigroups are among the main tools in this paper, they were introduced in \cite{MR0000000} Section 5.1.  Fix $(x_1,\cdots,x_n)\in\R^n$ and let $(X,B)$ be the Markov process associated to $Q^n$ and started from $(x_1,\cdots,x_n,0)$. In particular, $B$ is a Brownian motion. Write $X=(X^1,\cdots,X^n)$, then we will prove that for each $i$,
\begin{equation}\label{doub}
X^i_t=x_i + \int_{0}^{t}\text{sgn}(X^i_s) dB_s.
\end{equation}
Denote by $A$ the generator of $Q^2$ and let 
$$\mathcal D=\{f\in C^2(\R), f,f',f''\in C_0(\R),\ f'(0)=0\}.$$
Note that $\mathcal D$ is dense in $C_0(\R)$. Since $(K,W)$ solves $(T)$, It\^o's formula shows that for all $f\in\mathcal D$, $g\in C^2(\R)$ such that $g, g''\in C_0(\R)$, we have $f\otimes g\in\mathcal D(A)$ and  
$$A(f\otimes g)(x,w) = \frac{1}{2}\Delta(f\otimes g)(x,w) + f'(x) \text{sgn} (x) g'(w)$$
where $\Delta$ is the Laplacian on $\R^2$ and so  
\begin{equation}\label{input}
f(X^i_t) g(B_t) - \int_{0}^{t} A(f\otimes g)(X^i_s,B_s) ds
\end{equation}
is a martingale. By It\^o's formula again
\begin{equation}\label{output}
f(X^i_t) g(B_t) - \frac{1}{2}\int_{0}^{t} \Delta(f\otimes g)(X^i_s,B_s) ds - \int_{0}^{t} f'(X^i_s) g'(B_s) d\langle X^i,B\rangle_s
\end{equation}
is also a martingale. Thus the difference (\ref{input}) - (\ref{output}) is a martingale which is also a process of finite variation so it is identically zero, i.e.
$$\int_{0}^{t} (f'\text{sgn})(X^i_s) g'(B_s) ds = \int_{0}^{t} f'(X^i_s) g'(B_s) d\langle X^i,B\rangle_s.$$
An approximation argument shows that $\langle X^i,B\rangle_t = \int_{0}^{t}\text{sgn}(X^i_s) ds$ and consequently (\ref{doub}) holds in $L^2(\mathbb P)$. To finish the proof, recall that $B$ is an $(\mathcal F_t^{X,B})_t$ Brownian motion and since $\mathcal F^B_t\subset \mathcal F_t^{X}$ (from (\ref{doub})), we deduce that $B$ is an  $(\mathcal F_t^{X})_t$ Brownian motion.

(2) Let $(P^{n,c})_{n\ge 1}$ be the compatible family of Feller semigroups associated to $(P^n)_{n\ge 1}$ by Theorem 4.1 \cite{MR2060298} (note that condition $(C)$ there is satisfied). Let $X^n=(X^{n,1},\cdots,X^{n,n})$ be the Markov process associated to $P^n$ and started from $(x_1,\cdots,x_n)\in\R^n$ and let $W^n$ be an $(\mathcal F_t^{X^n})_t$ Brownian motion such that for each $i$, $(X^i,W^n)$ satisfy (\ref{der}). The Markov process $Y^{n}=(Y^{n,1},\cdots,Y^{n,n})$ associated to $P^{n,c}$ started from $(x_1,\cdots,x_n)$ is coalescing and the construction of $Y^n$ from $X^n$ shows that each $Y^{n,i}$ is solution of (\ref{oui}) driven by $W^n$ and in particular, 
\begin{equation}\label{animal}
\langle Y^{n,i},Y^{n,j}\rangle_t=\int_{0}^{t} \text{sgn}(Y^{n,i}_s) \text{sgn}(Y^{n,j}_s) ds
\end{equation}
for all $i,j$. By Theorem 4.2 and 2.1 in \cite{MR2060298}, it is possible to construct on the same probability space a joint realization $(K^1,K^2)$ where $K^1$ and $K^2$ are two stochastic flows of kernels satisfying $K^1\overset{law}{=}\delta_{\varphi}$, $K^2\overset{law}{=}K$ and such that for all $s\leq t, x\in \R,\ K^2_{s,t}(x)=E[K^1_{s,t}(x)|K^2]$ a.s.  We notice that Theorem 1.1 and Theorem 2.1 in \cite{MR2060298} remain valid with the additional assumption that the flows of kernels are jointly measurable with respect to $(s,t,x,\omega)$ (see the lines before Lemma 1.10 in \cite{MR2060298}). Now to conclude the proof of $(2)$, we only need to check that $K^1$ (or $\varphi$) is solution of $(T)$. Define 
$$W_{s,t}=\lim_{x\rightarrow+\infty, x\in\mathbb Q}(\varphi_{s,t}(x) - x).$$
Using (\ref{animal}), we easily prove that $W$ is a real white noise and $(\varphi,W)$ solves $(T)$.

\end{proof}
\section{Construction of flows associated to Tanaka's SDE}\label{tou}
The content of this section is taken from \cite{MR2235172}. We fix a probability measure $m$ on $[0,1]$ with mean $1/2$, then using Kolmogorov extension theorem one can construct on a probability space $(\Omega,\mathcal A,\mathbb P)$ a process $(U_{s,t},W_{s,t})_{s\leq t}$ indexed by $\{(s,t)\in\R^2, s\leq t\}$ taking values in $[0,1]\times \R$ whose law is characterized by 
\begin{itemize}
 \item [(i)] $W_{s,t}:=W_t-W_s, s\leq t$ and $(W_t)_{t\in\R}$ is a Brownian motion on the real line.
 \item[(ii)] For fixed $s<t$, $U_{s,t}$ is independent of $W$ and $U_{s,t}\overset{law}{=}m$.
\end{itemize}
Set for all $s<t, \textrm{min}_{s,t}=\inf\{W_{u}: u\in[s,t]\}.$ Then
\begin{itemize}
\item [(iii)] For all $s<t$ and $\{(s_{i},t_{i}); 1\leq i\leq n \}$ with $s_{i}<t_{i}$, the law of $U_{s,t}$ knowing
  $(U_{s_{i},t_{i}})_{1\leq i\leq n}$ and $W$ is given by $m$ when $\textrm{min}_{s,t}\not\in\{\textrm{min}_{s_{i},t_{i}};1\leq i\leq n\}$ and is given by

 $$\sum\limits_{i=1}^{n}\displaystyle{\delta_{U_{s_{i},t_{i}}}\times\frac{1_{\{\textrm{min}_{s,t}=\textrm{min}_{s_{i},t_{i}}\}}}
 {\textrm{Card}\{i;\textrm{min}_{s_{i},t_{i}}=\textrm{min}_{s,t}\}}}$$
otherwise.
\end{itemize}
Note that (i)-(iii) uniquely define the law of $(U_{s_{1},t_{1}},\cdots,U_{s_{n},t_{n}},W)$ for all $s_i<t_i, 1\leq i\leq n$. 

For $s, x\in \R$, define $$\tau_{s}(x)=\inf\{r\geq s:\ W_{s,r}=-|x|\}$$ and for $x\in\R, s\leq t$, let $W_{s,t}^{+}=W_{t}-\textrm{min}_{s,t}$,
$$K^{m}_{s,t}(x)=\delta_{x+\textrm{sgn}(x)W_{s,t}}1_{ \{t\leq \tau_{s}(x)\}}+(U_{s,t}\delta_{W_{s,t}^{+}}+(1-U_{s,t})\delta_{-W_{s,t}^{+}}) 1_{ \{t> \tau_{s}(x)\}}.$$
Note that for all $s\le t$, $K^m_{s,t}$ is jointly measurable with respect to $(x,\omega)$. Setting $\tilde{U}_{s,t}=\limsup_{n\rightarrow\infty} U_{s_n,t_n}$ with $(s_n,t_n)=(\frac{\lfloor ns \rfloor+1}{n},\frac{\lfloor nt \rfloor-1}{n})$, we get a version of $K^m$ which is measurable from $\{(s,t,x,\omega), s\leq t, x\in\R, \omega\in\Omega\}$ into $\mathcal P(\R)$. The new version $(K^m,W)$ is a solution of $(T)$. 
\section{Unicity of the Wiener flow}\label{tft}
In this section, we prove the unicty of the Wiener solution of $(T)$ and then discuss some extensions of the proof.
\begin{proposition}\label{ali}
Let $(K,W)$ be a Wiener solution of $(T)$. Then for all $s\le t, x\in\R$, with probability $1$, 
$$K_{s,t}(x)=\delta_{x+\textrm{sgn}(x)W_{s,t}}1_{ \{t\leq \tau_{s}(x)\}}+\frac{1}{2}(\delta_{W_{s,t}^{+}}+\delta_{-W_{s,t}^{+}}) 1_{ \{t> \tau_{s}(x)\}}$$

\end{proposition}
\begin{proof}
We will use the Feller semigroup
$$Q_t(f\otimes g)(x,w)=E[K_{0,t}f(x) g(w+W_t)].$$
Fix $x\in\R$ and $t>0$. Since $K$ is a Wiener solution, there exists a measurable function $F_{t,x}:C([0,t],\R)\rightarrow\mathcal P(\R)$ such that $K_{0,t}(x)=F_{t,x}(W_u, u\le t)$. Let $(X^x,B)$ be the Markov process associated to $Q$ and started from $(x,0)$. Then $B$ is a Brownian motion which we denote by $W$ and denote also $\tilde{K}_{0,t}(x)=F_{t,x}(B_u, u\le t)$ by $K_{0,t}(x)$ to simplify notations. We will prove the following: For all measurable bounded $f:\R\rightarrow\R$ a.s. 
\begin{equation}\label{yes}
K_{0,t}f(x) = E[f(X^x_t)|\mathcal F^W_{0,t}].
\end{equation}
We will check that for all $t_1\le\cdots\le t_{n-1}\le t_n=t$ and all bounded functions $f, g_1,\cdots,g_n$, we have 
\begin{equation}\label{rassoul}
E\big[K_{0,t} f(x)\prod_{i=1}^{n} g_i(W_{t_i})\big]=E\big[f(X^x_t)\prod_{i=1}^{n} g_i(W_{t_i})\big].
\end{equation}
This is easy to prove by induction on $n$. For $n=1$, (\ref{rassoul}) is immediate from the definition of $Q$. Let us prove the result for $n=2$. We have 
$$E[K_{0,t} f(x)g_1(W_{t_1})g_2(W_{t})] = E[K_{0,t_1}(Q_{t - t_1}(f\otimes g_2)(\cdot,W_{t_1})(x) g_1(W_{t_1})].$$
On the other hand 
$$E[f(X^x_t)g_1(W_{t_1})g_2(W_{t})] = E[Q_{t - t_1}(f\otimes g_2)(X^x_{t_1},W_{t_1})g_1(W_{t_1})].$$
Now the equality between both quantities holds using a uniform approximation of $Q_{t_2 - t_1}(f\otimes g)$ by a linear combination of functions of the form $h\otimes k$, $h, k\in C_0(\R)$. Let us derive the expression of $K_{0,t}(x)$. Proposition \ref{propop} (1) shows that $X^x$ is solution of (\ref{oui}) driven by $W$ with initial condition $X^x_0=x$. As $\text{sgn}(X^x_t)$ is independent of $W$ on the event $\{t>\tau_0(x)\}$ and $|X^x_t|=W^+_{0,t}$, the proof is finished.
\end{proof}
Since the law of $(X^x,W)$ is easy to describe here, we got an explicit expression of the Wiener solution. When the law of $(X^x,W)$ is unique with $X^x$ a weak solution of an SDE driven by $W$ and starting from $x$, the arguments above may be applied to prove that at least the Wiener solution to the generalized equation is unique whenever it exists. Let us discuss the example considered in \cite{MR000}.
\begin{proposition} (3.1 of \cite{MR000})
Given $W^-$ and $W^+$ two independent real white noises, there exists at most one Wiener flow $K$ such that for all $s\leq t, x\in\R$, $f\in C^2_b(\mathbb R)$, a.s.
\begin{eqnarray}\label{prtt}
K_{s,t}f(x)&=&f(x)+\int_s^tK_{s,u}(f'1_{]-\infty,0]})(x)dW^-_{s,u}\nonumber\\
&+&\int_s^tK_{s,u}(f'1_{]0,\infty[})(x)dW^+_{s,u}+\frac{1}{2}\int_s^t K_{s,u}f''(x)du.\nonumber\\ \label{popo}
\end{eqnarray}
\end{proposition}
\begin{proof}
Let $W^1$ and $W^2$ be two independent standard Brownian motions. For a given $x$, the SDE
\begin{equation}\label{dr}
dX^x_t=1_{\{X^x_t\le 0\}} dW^1_t + 1_{\{X^x_t>0\}} dW^2_t,\ X^x_0=x
\end{equation}
has a weak solution and moreover the law $\mathbb Q_x$ of $(X^x,W^1,W^2)$ is unique (see Proposition 4.1 in \cite{MR000}). Now let $K^1$ and $K^2$ be two Wiener solutions of (\ref{popo}), then 
$$K_{s,t}(x,y)=K^1_{s,t}(x)\otimes K^2_{s,t}(y)$$
is a stochastic flow of kernels on $\R^2$ and 
$$Q^2_t(f\otimes g\otimes h)(x,y,w)=E[K^1_{s,t}f(x)K^2_{s,t}g(y)h(w+ W_t)]$$
where $f,g\in C_0(\R), h\in C_0(\R^2)$ and $x,y\in\R, w\in\R^2$ defines a Feller semigroup on $\R^4$. Fix $s=0, t>0, x\in\R$, and let $(X,Y,B)$ be the Markov process associated to $Q$ started from $(x,x,0)$, then $B=(B^1,B^2)$ is a two dimensional Brownian motion and following the proof of (\ref{yes}), we have for all $f\in C_0(\R)$ a.s.
$$N^1_{0,t}f(x) = E[f(X_t)|\mathcal F^B_{0,t}],\ \ N^2_{0,t}f(x) = E[f(Y_t)|\mathcal F^B_{0,t}]$$
where $(N^1_{0,t}(x),N^2_{0,t}(x))$ is a copy of $(K^1_{0,t}(x),K^2_{0,t}(x))$. Using martingale problems as in the proof of Proposition \ref{propop} (1), we easily check that $(X,B)$ and $(Y,B)$ satisfy (\ref{dr}) and thus have law $\mathbb Q_x$. Consequently, for all $f\in C_0(\R)$, a.s.
$$E[f(X_t)|\mathcal F^B_{0,t}] = E[f(Y_t)|\mathcal F^B_{0,t}]$$
which yields that $K^1_{0,t}(x)=K^2_{0,t}(x)$ a.s. 
\end{proof}
The method described here applies also for the examples considered in \cite{MR2835247,MR50101111,MR0000000,MR501878} but there the notion of weak solution starting from a single point should be defined carefully (see Definition 1.1 in \cite{MR0000000}). 
\begin{remark}\label{mama}
This remark is a consequence of Exercise 3.13 on page 204 \cite{MR1725357}. Let $(X^x,W)$ be a weak solution of (\ref{oui}) and let $f:\R\rightarrow\R$ be measurable and bounded. Fix $t>0$ and define $\psi(s,x)=p_{t-s} f(x)$ for $0<s<t, x\in\R$ where $p$ is the semigroup of the standard Brownian motion. Applying It\^o's formula for the semimartingale $(s,X^x_s)$ and then letting $s\uparrow t$, we see that
$$f(X^x_t)=p_t f(x) + \int_{0}^{t} \left[(p_{t-u} f)' \text{sgn}\right] (X^x_u) dW_u.$$
Using (\ref{yes}), we get  
$$K_{0,t}f(x)=p_t f(x) + \int_{0}^{t} K_{0,u}((p_{t-u} f)'\text{sgn})(x) dW_u.$$
Iterating this relation, we obtain the Wiener chaos expansion of $K_{0,t} f(x)$. The same method works without difficulty for the SDE (\ref{prtt}) and the SDE considered in \cite{MR50101111}. For the SDEs studied in \cite{MR2835247,MR0000000}, one could generalize the previous idea by establishing first an It\^o's formula for $(t,X_t)$ where $X$ is a Walsh's Brownian motion and then proceeding as above.
\end{remark}
\section{Classification of weak solutions}\label{tr}
In this section, we present an easy proof of Theorem \ref{a} (3). So fix $(K,W)$ a solution of $(T)$. Since $\mathcal F^W_{s,t}\subset \mathcal F^K_{s,t}$ for all $s\le t$, we can define $\hat{K}$ a Wiener stochastic flow obtained by filtering $K$ with respect to $\sigma(W)$ (Lemma 3-2 (ii) in \cite{MR2060298}). Thus by conditioning with respect to $W$, we see that $(\hat{K},W)$ also solves $(T)$. By the result of the previous section $\hat{K}$ is therefore a modification of $K^W$. Thus setting $U_{s,t}=K_{s,t}(0,[0,\infty[)$, we deduce that for all $s\le t$ and $x\in\R$, a.s.
$$K_{s,t}(x)=\delta_{x+\textrm{sgn}(x)W_{s,t}}1_{ \{t\leq \tau_{s}(x)\}}+(U_{s,t}\delta_{W_{s,t}^{+}}+(1-U_{s,t})\delta_{-W_{s,t}^{+}}) 1_{ \{t> \tau_{s}(x)\}}.$$
Set $U_{t}:=U_{0,t}$ for $t>0$. It remains to prove the following 
\begin{proposition}
For all $t>0$, $U_{t}$ is independent of $(W_{u})_{u\ge 0}$ and the law of $U_{t}$ does not depend on $t>0$. Denote by $m$ the law of $U_{t}$ for $t>0$, then $K$ and $K^m$ have the same law.
\end{proposition}
The last claim is a direct consequence of the two first ones, since $K$ and $K^m$ will define the same compatible family of Feller semigroups $P^n_t=E[K^{\otimes n}_{0,t}]=E[{K_{0,t}^m}^{\otimes n}], n\ge 1$. Now the rest of this section will be devoted to the proof of this proposition. Denote by $m_t$ the law of $U_{t}$ for $t>0$. The key observation is the following
\begin{lemma}\label{immm}
The following assertions are equivalent
\begin{itemize}
\item[(i)] For all $t>0$, $U_{t}$ is independent of $(W_{u})_{u\ge 0}$ and the law of $U_{t}$ does not depend on $t>0$.
\item[(ii)] For all $n\ge 1$, if $(X^1,\cdots,X^n)$ is the $n$ point motion associated to $K$ started from $(0,\cdots,0)$, then $(\text{sgn}(X^1_t),\cdots,\text{sgn}(X^n_t))$ is independent of $|X^1|$ (which is also equal to any $|X^i|$) for all $t>0$ and the law of $(\text{sgn}(X^1_t),\cdots,\text{sgn}(X^n_t))$ does not depend on $t>0$.
\end{itemize}

\end{lemma}
\begin{proof}
For $n\ge 1$,  
$$Q^n_t(f\otimes g)(x,w)=E[K^{\otimes n}_{0,t}f(x) g(w+ W_t)],$$
$f\in C_0(\R^n), g\in C_0(\R), x\in\R^n, w\in\R$ defines a Feller semigroup on $\R^n\times\R$. Denote by $(X^1,\cdots,X^n,B)$ the Markov process associated to $Q^n$ and started from $(0,\cdots,0,0)$. An easy induction similar to the proof of (\ref{rassoul}) shows that for all $N\ge 1$ and all bounded continuous functions $f_1,\cdots,f_N:\R^n\rightarrow\R, g_1,\cdots,g_N:\R\rightarrow\R$, we have 
\begin{equation}\label{putry}
E\bigg[\prod_{i=1}^{N} f_i(X^1_{t_i},\cdots,X^n_{t_i})g_i(B_{t_i})\bigg]=E\bigg[\prod_{i=1}^{N} K_{0,t_i}f_i(0)g_i(W_{t_i})\bigg].
\end{equation}
By Proposition \ref{ali}, $X^i$ is a solution of Tanaka's SDE driven by $B$ and so $|X^i_t|=B^+_t:=B_t-\inf_{0\le u\le t} B_u$ for all $i$. Now (i) entails (ii) is clear using (\ref{putry}). Assume (ii), then by (\ref{putry}) $E[U_{t}^N g(W)]=E[U_{t}^N] E[g(W)]$ for each $N$ so that $U_{t}$ is independent of $W$ and similarly the law of $U_t$ does not depend on $t$.
\end{proof}
In the rest of this section $n\ge 1$ is fixed and $(X^1,\cdots,X^n)$ is the $n$ point motion associated to $K$. We will prove (ii) in Lemma \ref{immm} for $n$ which reduces to proving that for all $1\le i_1\le \cdots\le i_k\le n$, $\text{sgn}(X^{i_1}_t)\cdots\text{sgn}(X^{i_k}_t)$ is independent of $B$. This is sufficient since it yields that for all $k_1,\cdots,k_n$,
$$E[(\text{sgn}(X^{1}_t))^{k_1}\cdots(\text{sgn}(X^{n}_t))^{k_n} h(B)]$$
coincides with 
$$E[(\text{sgn}(X^{1}_t))^{k_1}\cdots(\text{sgn}(X^{n}_t))^{k_n}] E[h(B)]$$
for any measurale bounded $h:C(\R_+,\R)\rightarrow\R$. To simplify notations, we take $i_1=1,\cdots,i_k=k$. We will need the following lemma which is an easy consequence of the strong Markov property

\begin{lemma}
For any $\epsilon>0$, define the stopping times $\tau^{\epsilon}_0=0$, 
\begin{eqnarray}
\sigma^{\epsilon}_l&=&\inf\{u\ge \tau^{\epsilon}_l : B^+_u=\epsilon\},\nonumber\\
\tau^{\epsilon}_{l+1}&=&\inf\{u\ge \sigma^{\epsilon}_l : B^+_u=0\}.\nonumber\
\end{eqnarray}
Then, the sequence $(\text{sgn}(X^i_{\sigma^{\epsilon}_l}))_{1\le i\le k}, l=0,1,\cdots$ is i.i.d.  Moreover for any $l$ and $\epsilon'>0$, $(\text{sgn}(X^i_{\sigma^{\epsilon}_l}))_{1\le i\le k}$ and $(\text{sgn}(X^i_{T}))_{1\le i\le k}$ have the same law where 
$$T=\inf\{u\ge 0 : B^+_u=\epsilon'\}.$$
\end{lemma}
Now set $$\alpha_k=\mathbb P(\text{sgn}(X^1_{T})\cdots\text{sgn}(X^k_{T})=1)$$
where $T$ is a stopping time as in the previous lemma, then we have the following
\begin{proposition}
$Z_t=\text{sgn}(X^1_{t})\cdots\text{sgn}(X^k_{t}) B^+_t$ is a skew Brownian motion with parameter $\alpha_k$. In particular, $\text{sgn}(X^1_{t})\cdots\text{sgn}(X^k_{t})$ is independent of $|Z|=B^+$ for all $t>0$.
\end{proposition}
\begin{proof}
For $N\ge 1$, define 
$$T^N_0=0,\ T^{N}_{l+1}=\inf\{t>T^N_l : |X^1_t - X^1_{T^N_l}|=2^{-N}\}$$
and 
$$S^N_l = 2^N \text{sgn}(X^1_{T^N_l})\cdots\text{sgn}(X^k_{T^N_l}) B^+_{T^N_l},\ \ l=0,1,\cdots.$$
For every $N$, $(S^N_l)_l$ is a Markov chain started at $0$ the law whose law is given by the transition probabilities
$$Q(0,1)=1-Q(0,-1)=\alpha_k,\ \ Q(m,m+1)=Q(m,m-1)=1/2,\ m\ne 0.$$
In particular, $\left(2^{-N} S^N_{\lfloor 2^{2N}t\rfloor}\right)_{t\ge 0}$ converges in finite dimensional distributions as $N\rightarrow\infty$ to the skew Brownian motion with parameter $\alpha_k$ (see \cite{MR606993}). Now $\lim_{N\rightarrow\infty}T^N_{{\lfloor 2^{2N}t\rfloor}}=t$ a.s. uniformly on compact sets (see \cite{MRolm} page 31). Since $Z$ is continuous, $Z_{T^N_{{\lfloor 2^{2N}t\rfloor}}}$ converges to $Z_t$ a.s. uniformly on compact sets. But $Z_{T^N_{{\lfloor 2^{2N}t\rfloor}}}=2^{-N} S^N_{\lfloor 2^{2N}t\rfloor}$, so we deduce that $Z$ is a skew Brownian motion with parameter $\alpha_k$.

\end{proof}
By the fundamental result of \cite{MR606993}, $V_t=Z_t - (2\alpha_k-1) L_t(Z)$ is a standard Brownian motion where $L_t(Z)$ stands for the symmetric local time of $Z$ and $Z$ is also the unique strong solution to $Z_t=V_t + (2\alpha_k-1) L_t(Z)$. The next proposition gives more informations on the Brownian motion $V$ 
\begin{proposition}
We have  
 $$Z_t = \int_{0}^{t} \text{sgn}(X^1_{s})\cdots\text{sgn}(X^k_{s}) dB_s + (2\alpha_k-1) L_t(Z).$$
In particular, $Z_t$ and $V_t=\int_{0}^{t} \text{sgn}(X^1_{s})\cdots\text{sgn}(X^k_{s}) dB_s$ define the same filtrations.
\end{proposition}
\begin{proof}
The claim follows exactly the proof of Proposition 3 in \cite{MR090}. The slight difference here is that $Z$ is obtained by flipping independently the excursions of $B^+$ with an i.i.d sequence $\{\xi_l\}$ such that $\mathbb P(\xi_l=1)=1-\mathbb P(\xi_l=1)=\alpha_k$, so that $E[\xi_1]$ at the end of the proof of Proposition 3 \cite{MR090} will be replaced with $(2\alpha_k-1)$. 
\end{proof}
Note that since $E[ \text{sgn}(X^1_{t})\cdots\text{sgn}(X^k_{t})]=2\alpha_k-1$, we see from (\ref{putry}) that
$$\alpha_k=\frac{1}{2}\bigg(1+\int_{0}^{1} (2x-1)^k dm(x)\bigg).$$
This shows that the moments of $m$ up to the order $k$ are uniquely determined by the $k$ point motions.
\section{Generators of the $n$-point motions}\label{trrr}
In this section $K=K^m$ is a solution of Tanaka's SDE associated to $m$. Our purpose here is to write the generator of $P^n_t=E[K_{0,t}^{\otimes n}]$ on a core of $C_0(\R^n)$ and show its dependence on $m$. 

For all $\epsilon_1,\cdots,\epsilon_i \in \{-1,1\}$, set 
$$M_{\epsilon_1,\cdots,\epsilon_i}=\int_{0}^1 \alpha^{I} (1-\alpha)^{i-I} dm(\alpha)$$
where $I=\text{Card}\{j\in [1,i] : \epsilon_j=1\}$. Let $\mathcal D_m$ be the set of all functions $f:\R^n\rightarrow \R$ which are in $C_0(\R^n)$ and satisfy the following assumptions
\begin{itemize}
\item[(i)] For all $E=E_1\times\cdots\times E_n$, with $E_i=\R^{\ast}_+$ or $\R^{\ast}_-$, $f_{|E}$ is the restriction on $E$ of a $C^2$ function $g$ (=$g_E$) on $\R^n$ such that for all $i,j$, $\frac{\partial g}{\partial x_i},\frac{\partial^2 g}{\partial x_i\partial x_j}$ are in $C_0(\R^n)$ and for all $x\in\partial E$, 
$$\lim_{y\rightarrow x, y\in E} \frac{\partial^2 f}{\partial x_i\partial x_j}(y)=0. $$
\item[(ii)] For all $x=(x_1,\cdots,x_n)\in\R^n$ if $\{x_h, x_h=0\}=\{x_{i_1},\cdots,x_{i_j}\}$ with $i_1\le \cdots \le i_j$, then 

$$\sum_{\epsilon_{i_1},\cdots,\epsilon_{i_j}\in \{-1,1\}} M_{\epsilon_{i_1},\cdots,\epsilon_{i_j}} \sum_{k\in\{i_1,\cdots,i_j\}}\epsilon_k\lim_{h\rightarrow0+}\frac{\partial f}{\partial x_k}(y_h(x))=0$$
where the $l$-th coordinate of $y_h(x)$ is given by $\epsilon_l h$ if $l\in\{i_1,\cdots,i_j\}$ and by $x_l$ otherwise.

\end{itemize}
Note that $\mathcal D_m$ is dense in $C_0(\R^n)$ since it contains $\mathcal D\otimes\cdots\otimes \mathcal D$ with $\mathcal D=\{f\in C^2(\R), f,f',f''\in C_0(\R), f'(0)=0\}$. We have the following
\begin{proposition} The generator of $P^n$ coincides on $\mathcal D_m$ with $A^n$ where for $f\in\mathcal D_m$ and $x=(x_1,\cdots,x_n)\in\R^n$, $A^n f(x)$ is defined by: If $x_i\ne 0$ for all $1\le i\le n$, then 
$$A^nf(x):=\frac{1}{2} \sum_{h,k} \text{sgn}(x_h)\text{sgn}(x_k)\frac{\partial^2 f}{\partial x_h\partial x_k}(x).$$
If $\{x_h, x_h=0\}=\{x_{i_1},\cdots,x_{i_j}\}$ with $i_1\le \cdots \le i_j$, then 
$$A^nf(x_1,\cdots,x_n):=\frac{1}{2} \sum_{\epsilon_{i_1},\cdots,\epsilon_{i_j}\in \{-1,1\}} M_{\epsilon_{i_1},\cdots,\epsilon_{i_j}} \sum_{h,k\in\{i_1,\cdots,i_j\}} \epsilon_h\epsilon_k \frac{\partial^2 f}{\partial x_h\partial x_k}(x)$$
$$ + \frac{1}{2} \sum_{\epsilon_{i_1},\cdots,\epsilon_{i_j}\in \{-1,1\}} M_{\epsilon_{i_1},\cdots,\epsilon_{i_j}}  \sum_{h,k\notin\{i_1,\cdots,i_j\}}\text{sgn}(x_h)\text{sgn}(x_k)\frac{\partial^2 f}{\partial x_h\partial x_k}(x) $$
$$+\sum_{\epsilon_{i_1},\cdots,\epsilon_{i_j}\in \{-1,1\}} M_{\epsilon_{i_1},\cdots,\epsilon_{i_j}} \sum_{h\in\{i_1,\cdots,i_j\},k\notin\{i_1,\cdots,i_j\}} \epsilon_h\text{sgn}(x_k)\frac{\partial^2 f}{\partial x_h\partial x_k}(x).$$

\end{proposition}
\begin{proof}
Denote $\tau_0(y)$ by $\tau_y$. Since for all $y\ne 0$, $\lim_{t\rightarrow 0+} \mathbb P(t>\tau_y)/{t}=0$, using the fact that $U_t$ is independent of $W$ and has for law $m$, we have as $t\rightarrow 0+$,
$$E[K_{0,t}^{\otimes n}f(x)] = \sum_{\epsilon_1,\cdots,\epsilon_i\in \{-1,1\}} M_{\epsilon_1,\cdots,\epsilon_i} E[g_{\epsilon_1,\cdots,\epsilon_i}(W_{t\wedge\tau_y},W^+_{t\wedge\tau_y})] + o(t)$$
where $y\ne 0$ is fixed from now on such that $|y|< |x_i|$ for all $x_i\ne 0$ and $g_{\epsilon_1,\cdots,\epsilon_i}$ is defined on $\R^{\ast}\times \R^{\ast}_+$ by
$$g_{\epsilon_1,\cdots,\epsilon_i}(a,b)=f\big(C_1(a,b),\cdots,C_n(a,b)\big)$$
with $C_i(a,b)=\epsilon_i b$ if $i\in \{i_1,\cdots,i_j\}$ and $C_i(a,b)=x_{i} + \text{sgn}(x_{i}) a$ if not. Let $\tilde{f}$ be a $C^2$ extension of $f_{|E}$ as in (i) above where $E=E_1\times\cdots\times E_n$ and $E_i=\R^{\ast}_+$ if $x_i>0$ or $x_i=0$ and $\epsilon_i=1$, $E_i=\R^{\ast}_-$ if $x_i<0$ or $x_i=0$ and $\epsilon_i=-1$. By It\^o's formula applied to $\tilde{f}$, denoting $g=g_{\epsilon_1,\cdots,\epsilon_i}$, we have
\begin{eqnarray}
E[g(W_{t\wedge\tau_y},W^+_{t\wedge\tau_y})]&=&  E\bigg[\int_{0}^{t\wedge\tau_y}\bigg(\frac{1}{2}\Delta g + \frac{\partial^2 g}{\partial a\partial b}\bigg)(W_s,W^+_s) ds\bigg] \nonumber\\
&+& E\bigg[\int_{0}^{t\wedge\tau_y} \frac{\partial g}{\partial b}(W_s,0+) dL_s\bigg]\nonumber\
\end{eqnarray}
where $L_t=-\min_{0\le u\le t} W_u$ and $\Delta$ is the Laplacian on $\R^2$. Since $f\in\mathcal D_m$, by (ii) for all $s\le \tau_y$ we have    
$$\sum_{\epsilon_1,\cdots,\epsilon_i\in \{-1,1\}} M_{\epsilon_1,\cdots,\epsilon_i} \frac{\partial g_{}}{\partial b}(W_s,0+)=0,$$
and so $$\lim_{t\rightarrow0+} t^{-1}(E[K_{0,t}^{\otimes n} f(x)] - f(x))=\sum_{\epsilon_{i_1},\cdots,\epsilon_{i_j}\in \{-1,1\}} M_{\epsilon_{i_1},\cdots,\epsilon_{i_j}}\bigg(\frac{1}{2}\Delta g + \frac{\partial^2 g}{\partial a\partial b}\bigg)(0,0)$$
which is also equal to $A^n f(x)$.
\end{proof}
\bibliographystyle{plain}
\bibliography{Bil8}

\end{document}